\documentclass[12pt,oneside,reqno]{amsart}
\usepackage[all]{xy}
\usepackage{amsfonts,amsmath,oldgerm,amssymb,amscd}
\UseComputerModernTips
\numberwithin{equation}{section}
\usepackage[breaklinks]{hyperref}
\allowdisplaybreaks

\usepackage{enumerate}

\usepackage{caption} 
\newtheorem{theorem}{Theorem}[section]
\newtheorem{proposition}[subsection]{\bf Proposition}
\newtheorem{conj}{Conjecture}
\newtheorem{lemma}[subsection]{{\bf Lemma}}

\newtheorem{definition}[subsection]{Definition}

\newtheorem{remark}[subsection]{Remark}

\newcommand{\al}{\alpha}
\newcommand{\be}{\beta}
\newcommand{\ga}{\gamma}

\newcommand{\Z}{\mbox{$\mathbb Z$}}
\newcommand{\Q}{\mbox{$\mathbb Q$}}

\begin{document}

\title[On the resolution of the Diophantine equation $U_n + U_m = x^q$]{On the resolution of the Diophantine equation $U_n + U_m = x^q$} 

\author[P. K. Bhoi]{P. K.  Bhoi}
\address{Pritam Kumar Bhoi, Department of Mathematics, National Institute of Technology Rourkela-769008, India}
\email{pritam.bhoi@gmail.com}

\author[S. S. Rout]{S. S. Rout}
\address{Sudhansu Sekhar Rout, Department of Mathematics,
National Institute of Technology Calicut - 673 601, Kozhikode,  India.}
\email{sudhansu@nitc.ac.in; lbs.sudhansu@gmail.com}

\author[G. K. Panda]{G. K. Panda}
\address{Gopal Krishna Panda, Department of Mathematics, National Institute of Technology Rourkela-769008, India}
\email{gkpanda\_nit@rediffmail.com}

\thanks{2020 Mathematics Subject Classification: Primary 11B39, Secondary 11D61, 11J86. \\
Keywords: Binary recurrence sequence, Lucas sequences, Diophantine equation, Linear forms in logarithms}
\maketitle
\pagenumbering{arabic}
\pagestyle{headings}

\begin{abstract}
Suppose that $(U_{n})_{n \geq 0}$ is a binary recurrence sequence  and has a dominant root $\alpha$ with $\alpha>1$ and the discriminant $D$ is square-free. In this paper, we study the Diophantine equation $U_n + U_m = x^q$ in integers $n \geq m \geq 0$, $x \geq 2$, and $q \geq 2$. Firstly, we show that there are only finitely many of them for a fixed $x$ using linear forms in logarithms. Secondly, we show that there are only finitely many solutions in $(n, m, x, q)$ with $q, x\geq 2$ under the assumption of the {\em abc-conjecture}.  To prove this, we use several classical results like Schmidt subspace theorem, a fundamental theorem on linear equations in $S$-units and Siegel's theorem concerning the finiteness of the number of solutions of a hyperelliptic equation.
\end{abstract}

\section{Introduction}
The sequence $(U_{n})_{n \geq 0} = U_{n}(P, Q, U_{0}, U_{1})$ is called a binary linear recurrence sequence if the relation
\begin{equation}\label{eq4}
U_{n} = PU_{n-1} + QU_{n-2} \;\;(n\geq 2)
\end{equation}
holds, where $P, Q\in \mathbb{Z}$ with $PQ\neq 0$ and $U_{0}, U_{1}$ are fixed rational integers with $|U_{0}| + |U_{1}| > 0$.  The polynomial $f(x) = x^2 -Px -Q$ is called the companion polynomial of the sequence $(U_{n})_{n \geq 0}$. Let $D= P^2 + 4Q$ be the discriminant of $f$. We also call $D$ the discriminant of the sequence $(U_{n})_{n \geq 0}$. If the roots of the sequence of the companion polynomial are denoted by $\alpha$ and $\beta$, then for $n\geq 0$,
\begin{equation}\label{neweq5}
U_{n} = \frac{a_1\alpha^n - a_2 \beta^n}{\alpha-\beta} \quad (\alpha \neq \beta),
\end{equation}
where 
\begin{equation}\label{neweq5a}
a_1 = U_{1} - U_{0}\beta,\, a_2 = U_1- U_{0}\alpha.
\end{equation}
The sequence $(U_{n})_{n\geq 0}$ is called non-degenerate if $a_1a_2\al\be \neq 0$ and $\al/\be$ is not a root of unity. Throughout this paper, we assume that $(U_{n})_{n\geq 0}$ is non-degenerate, and the dominant root $\alpha$ is larger than one. We label the roots in such a way that $\alpha > |\beta|$ (that is $\alpha$ and $\beta$ are real and unequal). 

The Fibonacci sequence $(F_{n})_{n\geq 0}$ is an example of a non-degenerate binary recurrence sequence corresponding to $(P, Q)=(1, 1)$ and initial conditions $F_{0}=0$ and $F_{1}=1$. Recently, the Diophantine equation
\begin{equation}\label{eq1fib}
F_n \pm F_m = y^q,
\end{equation}
 where $n\geq m \geq 0, y\geq 2$ and $q\geq 2$ has been studied by a number of authors. For example, Luca and Patel \cite{lp} proved that if $n\equiv m \pmod 2$, then either $n\leq 36$ or $y =0$ and $n=m$. This problem is still open for $n \not \equiv m \pmod 2$. Using the {\em abc-conjecture}, Kebli et al. \cite{kkll} proved that there are only finitely many integer solutions $(n, m, y, q)$ of \eqref{eq1fib} with $y, q\geq  2$.  Pink and Ziegler \cite{Pink2016} computed all the nonnegative integer solutions $(n, m, z_1, \ldots, z_{46})$ of 
\[F_n + F_m = 2^{z_1} 3^{z_2}\cdots 199^{z_{46}}.\]  
Zhang and Togb\'e \cite{zt} studied a similar type of the Diophantine equation 
\[F_n^p \pm F_m^p =y^q,\] 
in positive integers $(n, m, y, p, q)$ with $q\geq 2 \; \hbox{and}  \;\gcd(F_n, F_m) = 1$. 

The Pell sequence $(P_{n})_{n\geq 0}$ is another example of non-degenerate binary recurrence sequence with $(P, Q)=(2, 1)$ and initial conditions $P_{0}=0$ and $P_{1}=1$.   Recently, Aboudja et al., \cite{ahrt} studied the Diophantine equation
\begin{equation}\label{eq1pell}
P_n \pm P_m = y^q,
\end{equation}
where $P_k$ is the $k$-th term of the Pell sequence. In particular, they find all solutions of \eqref{eq1pell} in positive integers $(n, m, y, q)$ under the assumption $n\equiv m \pmod 2$. Further, the Diophantine equations of the form
 \begin{equation}\label{eq9d}
U_n = x^{q} + R(x), 
\end{equation}
where $R(x)\in\Z[x]$ with degree $\deg R$ and height $H(R)$, has been studied by several authors. In this direction, Nemes and Peth\"{o} \cite{np84} showed that if $|Q|=1, H(R)<H$ and $\deg R \leq \min\{q(1-\gamma), q-3\}$, where $H$ and $\gamma<1$ are positive real numbers, then all integer solutions $(n, x, q)$ of \eqref{eq9d} with $ |x|>1, q\geq 2$ satisfy $\max\{n, |x|, q\}< c$, where $c$ is an effectively computable constant.  Nemes and Peth\"{o} \cite{np86} provided a necessary condition under which  \eqref{eq9d} has infinitely many solutions. In particular, they characterized $R(x)$ in terms of the Chebyshev polynomials. Furthermore, Peth\"{o} \cite{petho86} described solutions of  \eqref{eq9d} in $n$ also.

In this paper, we prove two theorems which generalize \eqref{eq1fib} and \eqref{eq1pell}. In particular, we study the Diophantine equation
 \begin{equation}\label{eq9a}
U_n + U_m = x^{q} \quad  \hbox{where} \quad n \geq m \geq 0, \quad x \geq 2, \quad q \geq 2, 
\end{equation}
and $U_k$ is the $k$-th term of a binary recurrence sequence. Several authors studied  the case $x =2$ of \eqref{eq9a}. Bravo and Luca studied the case when $U_{n}$ is the $n$-th Fibonacci number \cite[Theorem 2]{BL2015} and the $n$-th Lucas number \cite[Theorem 2]{Bravo2014}. Certain variations of the Diophantine equation \eqref{eq9a} has also been studied for binary recurrence sequences (see e.g. \cite{mr2019, mr}). Now we state one of our main result.

\begin{theorem}\label{th2}
 Suppose  $(U_{n})_{n \geq 0}$ is  a non-degenerate 
binary recurrence sequence of integers satisfying 
recurrence \eqref{eq4}. We assume that $(U_n)_{n\geq 0}$ has a dominant root $\alpha>1$ with  $\alpha>|\beta|$ and discriminant $D$ is square-free. Suppose $x \geq 2$ is a fixed integer. Then the nonnegative integer  solutions $(n, m, q)$  of the Diophantine equation \eqref{eq9a}  satisfy
\[ \max\{n, m, q\} < C_1 (\log x)^4 \]
whenever $n> C_2$, where $C_1 \;\hbox{and} \;C_2$ are  constants effectively computable in terms of $\al, \beta, U_0$ and $U_1$.
\end{theorem}

Secondly, we look at whether  equation \eqref{eq9a} has finitely or infinitely many solutions with $n\geq m\geq 0, x\geq 2$ and $q\geq 2$. Under the assumption of {\em abc-conjecture}, we prove that the Diophantine equation \eqref{eq9a} has only finitely many positive integer solutions $(n, m, x, q)$. Furthermore, we describe the cases when  \eqref{eq9a} may have infinitely many solutions (see Remark \ref{rem1}).

Let $T_n(x)$ denote the Chebyshev polynomial of degree $n$. The classical Chebyshev polynomials $T_n(x)$ satisfy the recurrence relation $T_{n+1}(x) = xT_{n}(x) -T_{n-1}(x)$ for $n\geq 1$ with $T_0(x) = 2, T_1(x) = x$ and we call this sequence as Chebyshev polynomial sequence. Note that under the assumptions $|Q|=1$, $\al > 1$ and non-degeneracy  of $(U_n)_{n\geq 0}$, it is easy to see that discriminant $D>0$ and $D$ is not a square.

\begin{theorem}\label{th3}
Let $(U_{n})_{n \geq 0}$ be a non-degenerate binary recurrence sequence of integers satisfying 
recurrence \eqref{eq4} with $|Q| = 1,  D\nmid a_1a_2$, where  $a_1, a_2$ defined in \eqref{neweq5a}. Further, assume that $(U_{n})_{n \geq 0}$ has a dominant root $\alpha >1$. Then under the assumption of abc-conjecture, the Diophantine equation \eqref{eq9a} has only finitely many nonnegative integer solutions in $(n, m, x, q)$ with $n \geq m, x \geq 2$ and $q \geq 2$.    
  \end{theorem}
\begin{remark}\label{rem1}
Let $U_0 =2, U_1 = P, Q  = \pm 1$, where $P$ is an integer and $P^2+4Q>0$. Then for all $n\in \Z$, $U_{2n} = U_n^2-2(-Q)^n$. For this sequence, we show that \eqref{eq9a} has infinitely many solutions $(n, m, x, q) = (4k, 0, U_{2k}, 2)$. Observe that, in this example, $D\mid a_1a_2$. Therefore, in Theorem \ref{th3} the assumption $D\nmid a_1a_2$ is necessary.

\end{remark}

\section{Auxiliary Results}\label{sec1}
In this section, we recall some existing definitions and results which will be needed in later sections to prove Theorem 1.1 and 1.2.
\begin{definition}
The radical of a positive integer $n$ is the product of the distinct prime numbers dividing $n$, i.e.,
\[\mbox{rad}(n)=\prod_{\substack{
    p\mid n\\
  p\;\; \mbox{prime} 
    }}p.\] 
\end{definition}
In 1980, Masser and Oesterl\'{e} formulated the following {\em abc-conjecture}. 
\begin{conj}[\cite{Masser1985}]\label{conj1}
For any given $\epsilon > 0$, there exists a positive constant $\kappa(\epsilon)$ such that if $a, b,$ and $c$ are coprime integers for which
\[a + b = c,\]
then
\begin{equation}\label{eq2.3}
\max(|a|,|b|, |c|)\leq \kappa(\epsilon)(\mbox{rad}(abc))^{1+\epsilon}.
\end{equation}
\end{conj}

Let $\eta$ be an algebraic number of degree $d$ with minimal polynomial over $\Z$ given by
\[a_{0}x^d + a_1x^d + \cdots + a_d = a_0 \prod_{i=1}^{d}\left(X - \eta^{(i)}\right)\]
where $a_0>0$ and the
$\eta^{(i)}$'s are the conjugates of $\eta$. The absolute {\it logarithmic height} of $\eta$ is defined as
\[
h(\eta) = \frac{1}{d} \left( \log a_0 + \sum_{i=1}^d \log \max ( 1, |\eta^{(i)}| ) \right). 
\]
In particular, if $\eta = p/q$ is a rational number with $\gcd(p, q) = 1$ and $q >0$, then $h(\eta) = \log \max\{|p|, q\}$.
Some important properties of logarithmic height which we use in the proof of our main results are described in the following lemma.
\begin{lemma}[\cite{Smart1998}] \label{lemheightprop}Let $\eta, \gamma$ be algebraic numbers and $t \in \mathbb{Z}$. Then
\begin{enumerate}[\normalfont(1)]
\item $h(\eta \pm \gamma) \leq h(\eta) + h(\gamma) + \log 2$,
\item $h(\eta\gamma^{\pm 1}) \leq h(\eta) + h(\gamma)$,
\item $h(\eta^{t}) = |t|h(\eta)$.
\end{enumerate}
\end{lemma}

To prove Theorem \ref{th2}, we use lower bounds for linear forms in logarithms to bound the index $n$ appearing in Eq.~\eqref{eq9a}. In particular, we need the following general lower bound for linear forms in logarithms due to Matveev \cite[Theorem 2.2]{Mat2000}.
\begin{theorem}\label{lem:matveev}
Let $\mathbb{L}$  be an algebraic number field and $d_{\mathbb{L}}$, the degree of $\mathbb{L}$. Let $\ga_1,\ldots, \ga_t \in \mathbb{L}$ be neither $0$ nor $1$ and let $b_{1},\ldots, b_{t}$ be nonzero rational integers. Let $A_{j}$ be positive integers such that 
\begin{equation}\label{eq8a}
A_j \geq \max \left\{ d_{\mathbb{L}}h(\ga_j) , |\log \ga_j|, 0.16  \right\}, \quad j= 1, \ldots,t.
\end{equation}
 Assume that $B\geq \max\{|b_1|, \ldots, |b_{t}|\}$ and $\Lambda:=\ga_{1}^{b_1}\cdots\ga_{t}^{b_t} - 1$. If $\Lambda \neq 0$ and $\mathbb{L}\subset \mathbb{R}$, then
\[|\Lambda| \geq \exp \left( -1.4\times 30^{t+3}\times t^{4.5}\times d_{\mathbb{L}}^{2}(1 + \log d_{\mathbb{L}})(1 + \log B)A_{1}\cdots A_{t} \right).\]
\end{theorem}

The following lemma is due to Sanchez and Luca \cite[Lemma 7]{sl}. 
\begin{lemma}[\cite{sl}]\label{lem9}
If $m\geq 1, T> (4m^2)^m$ and $T> \frac{x}{(\log x)^m}$, then
\[x <2^m T (\log T)^m.\]
\end{lemma}

The proof of Theorem \ref{th3} also depends on the subspace theorem \cite{schmidt}. Let $K$ be an algebraic number field and $\mathcal{O}_K$ be its ring of integers. Let $M_K$ be the collection of places of $K$. For $v\in M_K, x\in K$, we define the absolute value $|x|_v$ by 
\begin{enumerate}
\item $|x|_v = |\sigma(x)|^{1/[K:\Q]}$ if $v$ corresponds to the embedding $\sigma: K\hookrightarrow \mathbb{R}$,
\item $|x|_v = |\sigma(x)|^{2/[K:\Q]}$ if $v$ corresponds to the pair of complex embeddings $\sigma, \overline{\sigma}: K\hookrightarrow \mathbb{C}$,
\item $|x|_v =(N\mathfrak{p})^{-\mbox{ord}_{\mathfrak{p}}(x)/[K:\Q]}$ if $v$ corresponds to the prime ideal $\mathfrak{p}$ of $\mathcal{O}_K$.
Here $N\mathfrak{p} = \# (\mathcal{O}_K/\mathfrak{p})$ is the norm of $\mathfrak{p}$ and  $\mbox{ord}_{\mathfrak{p}}(x)$ the exponent of $\mathfrak{p}$ in the prime ideal decomposition of $(x)$, with $\mbox{ord}_{\mathfrak{p}}(0) = \infty$.
\end{enumerate}
These absolute values satisfy the product formula 
\begin{equation}\label{eqprodcut}
\prod_{v\in M_K}|x|_v =1 \quad \mbox{for}\quad x\in K\setminus\{0\}.
\end{equation}
Observe that $H(x) = |x|$ for $x\in \mathbb{Z}$. Further, we define
\[H(x) = \prod_{v\in M_K}\max \{1, |x|_v\}\] for $x\in K$ and that 
\[H_{L}(x) = H_K(x)^{[L:K]},\] for $x\in K$ and for a finite extension $L$ of $K$. Let ${\bf x} = (x_1, \ldots, x_n)\in K^n$ with ${\bf x}\neq 0$ and for $v\in M_K$, let
\[|x|_v = \max_{1\leq i\leq n} |x_i|_v.\] Then $H({\bf x})$, the {\em height} of ${\bf x}$ is defined as
\[H({\bf x}) = \prod_{v\in M_K}\max \{1, |x|_v\}.\]
 
 \begin{theorem}[Subspace Theorem]\label{thmsub} Let $K$ be an algebraic number field and let $S\subset M_K$ be a finite set of absolute values which contains all the infinite ones. For $v\in S$, let $L_{1, v}, \ldots, L_{n, v}$ be $n$ linearly independent linear forms in $n$ variables with coefficients in $K$. Let $\delta>0$ be given. Then the solutions of the inequality 
 \begin{equation}
 \prod_{v\in S}\prod_{i=1}^n |L_{i, v}({\bf x})|_v < H({\bf x})^{-\delta}
 \end{equation} 
 with ${\bf x}\in (\mathcal{O}_K)^n$ and ${\bf x}\neq 0$ lie in finitely many proper subspaces of $K^n$.
 \end{theorem}
 
 Let $S$ be a finite set of places of $K$ containing all the infinite places. Define the group of $S$-units of $K$ by
\begin{equation}\label{eq10a}
\mathcal{O}_S^{*} :=  \{x \in K : |x|_v =1 \quad \mbox{for}\quad v \in M_K\setminus S\}. 
\end{equation} We also need the following result on linear equations in $S$-units (see \cite{ev, sch}).

\begin{theorem}\label{prop4}
Let $\gamma_1, \ldots, \gamma_m \in K\setminus\{0\}$. Then the equation 
\[\gamma_1y_1+\cdots+\gamma_my_m=1\]
has finitely many solutions in $S$-units $y_1, \ldots,y_m$ such that no proper sub-sum \[\gamma_{i_1}y_{i_1}+ \cdots+\gamma_{i_r}y_{i_r} = 0.\] 
\end{theorem}
  Finally we recall the basic classical theorem due to C. L. Siegel concerning the finiteness of the number of solutions of a hyperelliptic equation.
  \begin{theorem}[\cite{Siegel}] \label{siegel}
Let $K$ be any number field and $\mathcal{O}_K$ the ring of its algebraic integers. Let $f(x)\in K[x]$ be a non-constant polynomial having at least $3$ roots of odd multiplicity. Then the Diophantine equation \[y^2 = f(x)\] has only finitely many integer solutions $(x, y)\in \mathcal{O}_K^2$.
 \end{theorem}
 
 Now we will find an upper bound for $|U_n+ U_m|$. This bound will be utilized in the proof of Theorem \ref{th2}.

\begin{lemma}\label{lem-upperbound}
There exist constants $c_1, d_1$  and $d_2$ such that the following inequalities hold:
\begin{enumerate}
\item $|U_n + U_m| < c_1 \alpha^n$,
\item if  $n> d_2$, then $q< d_1n$.
\end{enumerate}
\end{lemma}

\begin{proof}
Since $\alpha >|\beta|$ and $|\alpha- \beta| = \sqrt{D}$, from \eqref{neweq5}, we have
\begin{equation}\label{sec2:eq2}
|U_n| = \left|\frac{a_1\alpha^n-a_2\beta^n}{\alpha-\beta}\right| \leq  \frac{\alpha^n}{\sqrt{D}} \left(|a_1|+ |a_2|\frac{|\be|^n}{\alpha^n}\right)< c_0\alpha^n,
\end{equation}
where $c_0:= \frac{|a_1|+|a_2|}{\sqrt{D}}$.
Thus, 
\[|U_n + U_m| \leq |U_n| +|U_m| <c_0(\alpha^n+\alpha^m),\] and since 
$n\geq m$, we have 
\[|U_n + U_m| < c_1\alpha^n,\] where $c_1 := 2c_0$.
Using \eqref{eq9a}, we get
\[x^q = |U_n + U_m| < c_1\alpha^n.\] Now, by taking logarithms in the above inequality, we find
\[q\log x < \log c_1 + n\log \alpha = n \log \alpha\left(1 + \frac{\log c_1}{n\log \alpha}\right).\] Putting $d_2:= \max\left\{0, \frac{\log c_1}{\log \alpha}\right\}$ and $d_1:= 2\log \alpha$ and assuming that $n> d_2$, we obtain
\[q< d_1 n/\log x <d_1 n.\]
\end{proof}

\begin{lemma}\label{lem-lowerbound}\cite{Pink2016}
There exist constants $d_3$ and $d_4$ such that for $n>d_3$, we have
\[|U_n+U_m| > d_4\alpha^n.\]
\end{lemma}
From Lemma \ref{lem-lowerbound} and \eqref{eq9a}, we deduce $d_4\alpha^n< x^q$ and this implies
\begin{equation}\label{eq121}
n/q < d_5 \log x.
\end{equation}

\begin{lemma}\label{sec3: lem4}
If $(n, m, x, q)$ is a solution of \eqref{eq9a} with $n>m, x\geq 2, q\geq 2$, then for $n>d_6$ we have
\[n-m \leq   c_2 \log q \log x,\]
where the constant $d_6, c_2$  are effectively computable in terms of $\al, \beta, a_1 \;\hbox{and} \;a_2$.
\end{lemma}
\begin{proof}
Using \eqref{neweq5} and \eqref{sec2:eq2}, one can write \eqref{eq9a} as
 \begin{align*}
 \left| \frac{a_1\al^n}{\sqrt{D}} - x^q \right| = \left|U_m + \frac{a_2\be^n}{\sqrt{D}} \right| \leq |U_m| +\left| \frac{a_2\be^n}{\sqrt{D}} \right| \leq c_0\alpha^m + \frac{|a_2||\be|^n}{\sqrt{D}}.
 \end{align*}
Dividing both sides by $|a_1|\alpha^n/\sqrt{D}$, we get \begin{align}\label{sec3:eq1}
\begin{split}
\left|  1 - x^{q} \alpha^{-n} a_1^{-1}\sqrt{D} \right|  & \leq \frac{c_0\sqrt{D}\al^m} {|a_1|\al^n}  + \frac{|a_2|}{|a_1|}\frac{|\be|^n} {\al^n} \leq \frac{c_0\sqrt{D}}{|a_1|}\alpha^{m-n}+  \frac{|a_2|}{|a_1|}\left(\frac{\alpha}{|\beta|}\right)^{m-n}\\
&\leq \frac{c_0\sqrt{D}}{|a_1|}\left(\frac{|\beta|}{\alpha}\right)^{n-m}+  \frac{|a_2|}{|a_1|}\left(\frac{|\beta|}{\alpha}\right)^{n-m}\\
&\leq \left(\frac{c_0\sqrt{D}}{|a_1|}+  \frac{|a_2|}{|a_1|}\right)\left(\frac{|\beta|}{\alpha}\right)^{n-m}
\end{split}
\end{align}
for the case $|\beta|>1$, and for $|\beta|\leq 1$, 
\[\left|  1 - x^{q} \alpha^{-n} a_1^{-1}\sqrt{D} \right|   \leq \frac{c_0\sqrt{D}} {|a_1|} \al^{m-n} + \frac{|a_2|}{|a_1|}.\]
Setting $n =  qk + r$, where $ r \in \{0, 1, \ldots, q-1\}$, in the left hand side of \eqref{sec3:eq1}, we get
\begin{equation}\label{sec3:eq2}
 \left|  1 - x^{q} \alpha^{-qk-r} a_1^{-1}\sqrt{D} \right|   \leq  \left(\frac{c_0\sqrt{D}}{|a_1|}+  \frac{|a_2|}{|a_1|}\right)\max\left\{\frac{1}{\al^{n-m}}, \left(\frac{|\beta|}{\al}\right)^{n-m}\right\}. 
 \end{equation}	
Let \[\Lambda_1:= (x/\alpha^{k})^q  \alpha^{-r} a_1^{-1}\sqrt{D} - 1.\] Now we apply Theorem \ref{lem:matveev}, by taking \[\mathbb{L}= \Q(\alpha), \ga_1:= x/\alpha^{k}, \ga_2:= \alpha , \ga_3:= a_1^{-1} \sqrt{D}, b_1:= q, b_2:= -r, b_3:= 1 , t= 3 , d_{\mathbb{L}} = 2.\] 
If $\Lambda_1 =0$, then 
\begin{equation}\label{reveq001}
 a_1\alpha^{n} = x^{q} \sqrt{D}.
 \end{equation}
  Let $\sigma$ be a Galois automorphism of $\Q(\alpha)$ which sends $\sqrt{D}$ to $-\sqrt{D}$. Then applying $\sigma$ to \eqref{reveq001}, we get
  \begin{equation}\label{reveq002}
a_2 \beta^{n} = -x^{q} \sqrt{D}.
 \end{equation}
Hence, from \eqref{reveq001} and \eqref{reveq002}, we have $|a_1||\alpha|^n = |a_2||\beta|^n$ and this implies $n= d_6$ where $d_6:= \frac{\log (|a_2|/|a_1|)}{\log (|\al|/|\beta|)}$.
 But this is not true since $n>d_6$ and therefore, we may assume that $\Lambda_1 \neq 0$. We may take $B = q$. Then, using Lemma \ref{lemheightprop}, we get
\begin{align*}
h(\ga_1)&= h(x/\alpha^k) \leq h(x)+ h(\alpha^k)\\
&\leq  \log x + k h(\alpha)\leq \log x+ k \log \alpha.
\end{align*}
Again,
\[h(\ga_2) = \frac{1}{2} \left[ \log \max \left(1,\alpha\right) + \log \max \left(1,|\beta|\right)\right]\leq \log \alpha\] and $h(\ga_3) \leq h(\sqrt{D})+\log|a_1| =\log|a_1|+  \frac{\log D}{2}$. So, we take 	 \[A_1 =  k\log \al + \log x,\quad  A_2= 2\log \al, \quad A_3 =  2\log |a_1|+\log D.\] 
Since $x$ is fixed,  from \eqref{eq121} and Lemma \ref{lem-upperbound}, we obtain $d_7<n/q<d_8$. Using Theorem \ref{lem:matveev}, we get
\begin{equation}\label{eq123}
 |\Lambda_1|> \exp\left(-1.4 \cdot 30^{t+3}\cdot  t^{4.5} \cdot d_{\mathbb{L}}^2 (1+ \log d_{\mathbb{L}}) (1+ \log B)A_1 A_2 A_3\right).
\end{equation}
Comparing the upper bound of inequality \eqref{sec3:eq2} with the lower bound of $|\Lambda_1|$  in \eqref{eq123}, we obtain
\begin{align}\label{eq122}
\begin{split}
(n-m)\max\left\{\al, \left(\frac{\alpha}{|\beta|}\right)\right\} &< 1.4 \cdot 30^{t+3}\cdot  t^{4.5} \cdot d_{\mathbb{L}}^2 (1+ \log d_{\mathbb{L}}) (1+ \log B)A_1 A_2 A_3\\
&< 1.4 \cdot 30^{6}\cdot  3^{4.5} \cdot 4 (1+ \log 2) (1+ \log q)\\
&\times (k\log \al + \log x)(2\log \alpha) (\log D)+\log \left(\frac{c_0\sqrt{D}}{|a_1|}+  \frac{|b_1|}{|a_1|}\right).
\end{split}
\end{align}
As $k\leq n/q$, from \eqref{eq122}, it follows that
\[(n-m) <  c_2 \log q \log x,\]
where $c_2$  is effectively computable in terms of $\al, \beta, a_1, a_2$. 
\end{proof}
Now we are ready to prove Theorem \ref{th2}. Our proof closely follows the proof in \cite{mr2019, mr}.

\section{Proof of Theorem \ref{th2}}\label{sec2}

 Since $n\geq m$, we divide the proof into two parts according as $n=m$ and $n>m$. 

{\sc Case I.} ($n =m$).  In this case, \eqref{eq9a} can be written as
\begin{equation}\label{proofthm1eq1a}
2U_n = x^q.
\end{equation}
Thus, in view of Peth\"{o}  \cite[p.6]{petho}, \eqref{proofthm1eq1a} has only finitely many effectively computable solutions in integers $n, x, q$ with $|x|, q\geq 2$.

{\sc Case II.} ($n > m$). Using \eqref{neweq5}, one can rewrite \eqref{eq9a} as
\[ \left| \frac{a_1\al^n +a_1 \al^m}{\sqrt{D}} - x^q \right| = \left|\frac{a_2\be^n + a_2\be^m}{\sqrt{D}}\right|.\]
Since $n>m$, $|\beta|^n+ |\beta|^m < 2|\beta|^n$, and consequently
\begin{equation}\label{proofthm1eq1}
\left|  1 - x^{q} \sqrt{D}\alpha^{-n} a_1^{-1}(1+\al^{m-n})^{-1} \right|   \leq  \frac{|a_2|}{|a_1|}\frac{|\be|^n + |\be|^m}{\al^n(1+\al^{m-n})} \leq \frac{|a_2|}{|a_1|}\frac{2|Q|^n}{\al^{n+m}}. 
\end{equation}
Let
\[ \Lambda_2 =  x^{q} \sqrt{D}\alpha^{-n} a_1^{-1}(1+\al^{m-n})^{-1} - 1.\]
Now we employ Theorem \ref{lem:matveev} on $\Lambda_2$. To do so, we take 
\begin{align*}
&\mathbb{L} = \Q(\alpha), \ga_1:= x, \ga_2:= \alpha , \ga_3:= \sqrt{D}a_1^{-1}(1+\al^{m-n})^{-1};\\
&b_1:= q, b_2:= -n, b_3:= 1 , t= 3 , d_{\mathbb{L}} = 2.
\end{align*}
Assuming $n> d_2$, we may take $B = d_1n$ (see Lemma \ref{lem-upperbound}). If $\Lambda_2 = 0$, then we get 
\begin{equation}\label{proofthm1eq2}
x^{q} \sqrt{D} =a_1(\al^{n} + \al^{m}).
\end{equation}
Let $\sigma$ be a Galois automorphism of $\mathbb{L}$ which sends $\sqrt{D}$ to $-\sqrt{D}$. Then applying $\sigma$ to \eqref{proofthm1eq2}, we get   
\begin{equation}\label{proofthm1eq3}
-x^{q} \sqrt{D} =a_2(\be^{n} + \be^{m}).
\end{equation}
Since $\al>1$, by virtue of \eqref{proofthm1eq2} and \eqref{proofthm1eq3}, we have
\begin{equation}
|a_1|\al^n \leq |a_1|(\al^m +\al^n)  = |a_2(\be^m +\be^n)|.
\end{equation} If $|\beta|\leq 1$, then $|a_1|\al^n <2|a_2|< 2(|a_2|+|a_1|)$ which implies that $n < \frac{\log (2(|a_2|+|a_1|)/|a_1|)}{\log \alpha}$. Furthermore, if $|\beta|>1$, then
\[|a_1|\al^n \leq 2|a_2||\beta|^n< 2|\beta|^n(|a_2|+|a_1|)\]  and hence
\[n< \frac{\log (2(|a_2|+|a_1|)/|a_1|)}{\log (\alpha/|\beta|)}.\] Thus,
 \[n<c_4: = \max\left\{\frac{\log (2(|a_2|+|a_1|)/|a_1|)}{\log \alpha}, \frac{\log (2(|a_2|+|a_1|)/|a_1|)}{\log (\alpha/|\beta|)}\right\}.\] But this is not true since  $n  >c_4$ and therefore, $\Lambda_2\neq 0$. To apply Theorem \ref{lem:matveev}, we need to estimate $h(\ga_i)$ for $i=1, 2, 3$. One can easily see that $h(\ga_1) = \log x$, and $h(\ga_2) \leq \log \al$.  In view of Lemma \ref{sec3: lem4} and Lemma \ref{lemheightprop}, we have
\begin{align*}
h(\ga_3) &\leq h(\sqrt{D}) + h(a_1(1+\al^{m-n})) \leq \log \sqrt{D}+\log |a_1|+ \frac{(n-m)\log \al }{2}+ \log 2 \\
&\leq   c_5' \log x \log q.
\end{align*}
Hence,  
\[A_1 = 2\log x,\quad A_2= \log \al, \quad 	A_3 =   c_5 \log x \log q,\]
where $c_5$ is effectively computable in terms of $\al, \beta, a_1$ and $a_2$. Thus,
\begin{align*}
|\Lambda_2| > \exp[(-1.4)\cdot &30^6\cdot 3^{4.5}\cdot 4(1+\log 2)\\
&(1+\log (d_1n))\cdot 2\cdot (\log x)\cdot ( \log \al)  c_5 \log x \log q].  
\end{align*}
Now comparing the lower and upper bounds for $\Lambda_2$, we find
\begin{equation}\label{proofthm1eq5}
(n+m)\log \al +\log (|a_1|/|a_2|)-\log 2 -n\log |Q| < c_6\cdot (\log n)(\log x)^2 \log q.     
\end{equation}
Adding $(n-m)\log \al$ to both sides of \eqref{proofthm1eq5} and then using Lemma \ref{sec3: lem4} under the assumption $n>d_6$, we get
\begin{align*}
 n(2\log \al &- \log |Q|) \\
 &< \log 2 -\log (|a_1|/|a_2|)+   (n-m) \log \al +  c_6 (\log n)(\log x)^2 \log q \\
&<   \log 2 -\log (|a_1|/|a_2|) +  c_2 \log q \log x \log \al  +  c_6 (\log n)(\log x)^2 \log q \\
&< c_7 \log q \log x + c_6 (\log n)(\log x)^2 \log q. 
\end{align*}
Now an application of Lemma \ref{lem-upperbound} gives, $\log q < c_{8}\log n$. Hence,
\begin{align*}
n(2\log \al - \log |Q|)	&< c_9 \log n \log x + c_{10} (\log n)^2(\log x)^2  \\
&< c_{11} (\log n)^2(\log x)^2
\end{align*}
and consequently, 
\begin{equation}\label{proofthm1eq6}
\frac{n}{(\log n)^2} < c_{12} (\log x)^2.
\end{equation}
In view of Lemma \ref{lem9},  if $T > 2^8$ and $\frac{n}{(\log n)^2} < T$,   then $ n < 4T(\log T)^2$. Applying this inequality to \eqref{proofthm1eq6}, we get
\[ n < 4c_{12} (\log x)^2 (\log (c_{12} (\log x)^2 ))^2.\]
From the inequality $\log x < x $, it follows that for $n> C_2:= \max\{d_2, d_6, c_4\}$,
	                   \[ n < C_1 (\log x)^4,\]
where the constants $C_1, C_2$ are effectively computable in terms of $\al, \beta, a_1, a_2$. This completes the proof of the theorem. \qed

\section{Proof of Theorem \ref{th3}}\label{sec3}

 To prove Theorem \ref{th3}, we first show that when $q> C_3$, \eqref{eq9a} has only finitely many positive integer solutions if the {\em abc-conjecture} is true. We then show that when $2\leq q \leq C_3$, \eqref{eq9a} has infinitely many positive integer solutions  and further these solutions satisfy an equation of the form $am+ bn = c$ with $\max\{|a|, |b|\} < C_4$. Finally, using Theorem \ref{siegel}, we get a contradiction. Our proof closely follows the proofs in (\cite{fuchs2009, kkll}).

\begin{proposition}\label{propo1}
Let  $(U_{n})_{n \geq 0}$ be a non-degenerate 
binary recurrence sequence of integers satisfying 
recurrence \eqref{eq4} with $|Q| = 1$. Assume that $(U_{n})_{n \geq 0}$ has a dominant root $\alpha >1$ and suppose that the abc-conjecture holds. Then \eqref{eq9a} has only finitely many positive integer solutions $(n, m, x, q)$ with $n > C_5, \; q \geq 2$ and $n - m < C_6n$. In particular, it has only finitely many solutions when $q > C_3$.
 \end{proposition}
\begin{proof}
Suppose that $n>C_5$ and $q > C_3$. From \eqref{sec2:eq2}, we have $|U_n| < c_0\alpha^n$. Thus,\[ x^q= |U_n + U_m| \leq c_0(\al^n + \al^m)\] and so,
\begin{equation}\label{sec4:eq1}
q\log x \leq \log c_0 + \log (\al^n + \al^m) .
\end{equation}
Now using \eqref{sec4:eq1} in Lemma \ref{sec3: lem4}, we get
 \begin{align*}
 n-m  &\leq   c_2\log x \log q \\
      &\leq   c_2 \left(\frac{\log c_0 + \log (\al^n + \al^m) }{q}\right)\log q \\
      & \leq  c_2' \frac{\log q}{q} (\log c_0 + n \log \al).
 \end{align*}
Since the function $\frac{\log q}{q}$ is decreasing for $q>e$, for sufficiently large $q >  C_3$, we have
 \[ n-m \leq c_{14}n,\]
 where $c_{14}< 10^{-5}$. Thus, we take $C_6:= c_{14}$. We rewrite left hand side of \eqref{eq9a} as 
 \[U_n + U_m = \frac{(1+\al^{m-n})a_1\al^n - (1+\be^{m-n})a_2\be^n}{\sqrt D}.\]
Now, consider the integer 
  $$X_{n,m} := (1+\al^{m-n})a_1\al^n + (1+\be^{m-n})a_2\be^n.$$
 Observe that 
 \begin{equation}\label{eq101}
 X_{n,m}^2 - D(U_n + U_m)^2 = 4(-1)^na_1a_2Q^n(1+\al^{m-n})(1+\be^{m-n}) =: Y_{n,m}.
 \end{equation}
 Therefore,
  \begin{align*}
 |Y_{n,m}| &= |(-1)^{n} 4a_1a_2Q^n(1+\al^{m-n})(1+\be^{m-n})| = \left|4a_1a_2Q^n\left(1+ \frac{1}{\al^{n-m}}\right)\left(1+\frac{1}{\be^{n-m}}\right)\right|\\
 &=\left|4a_1a_2Q^n\frac{(\al^{n-m}+1)(\be^{n-m}+1)}{(\al\be)^{n-m}}\right| \leq 8a_1a_2|Q|^m(\al^{c_{14}n}+1) \leq c_{15}\al^{c_{14}n},
 \end{align*} since $|Q| = 1$.
 If
 \begin{equation}\label{eq102}
 d = \gcd(X_{n,m}^2 , D(U_n + U_m)^2),
 \end{equation}
 then it follows that $d | Y_{n, m}$ and $d\leq c_{16}\al^{c_{14}n}$.
 Applying {\em abc-conjecture}  to
 \[ \frac{Y_{n,m} }{d} + \frac{D(U_n + U_m)^2 }{d} = \frac{X_{n,m}^2}{d},\]
 we get
  \begin{align*}
  \frac{|X_{n,m}|^2}{d} &\leq c_{17}(\epsilon) \left(\mbox{rad}\left( \frac{|Y_{n, m}|D x^{2q}|X_{n,m}|^2 }{d^3}\right)\right)^{1+\epsilon} \leq c_{18}(\epsilon)\left(|X_{n,m}| x \alpha^{c_{14}n}\right)^{1+\epsilon}.
  \end{align*}
  This implies
\[|X_{n,m}|^2 \leq c_{18}(\epsilon)(|X_{n,m}|x\al^{c_{14}n} )^{1+\epsilon} d \leq c_{19}(\epsilon)(|X_{n,m}|^{1+\epsilon}x^{1+\epsilon}\al^{c_{14}n(2+\epsilon)}).\]
 Hence, for $\epsilon \in (0,1)$, it follows that
   $$\al^{(1-\epsilon)n}\leq |X_{n,m}|^{1-\epsilon} \leq c_{20}(\epsilon)  \al^{c_{14}(2+\epsilon)n+(1+\epsilon)(n/q)},$$ 
  as $x \leq c_{18}'\al^{n/q}$ and  $|X_{n,m}| > c_{19}'\al^n$.  Thus, 
  \begin{equation}\label{sec4:eq2}
  \al^{c_{22}(\epsilon)n} \leq c_{21}(\epsilon),
  \end{equation}
 where $c_{22}(\epsilon) := ((1-\epsilon)-\frac{1+\epsilon}{q}-c_{14}(2+\epsilon))$. Choose $\epsilon$ in such a way that $c_{22}(\epsilon) >1/2$ for $q\geq 2$. Then from \eqref{sec4:eq2}, we get
 \[n < \frac{\log (c_{21}\epsilon)}{c_{22}(\epsilon)\log \alpha} =: c_{23}.\]
 \end{proof}
 
Next we find  the solutions of \eqref{eq9a} in the range $ q\leq 2 \leq  C_3$ and  $n - m > C_6n$.
 
 \begin{proposition}\label{propo2}
Let  $(U_{n})_{n \geq 0}$ be a non-degenerate 
binary recurrence sequence of integers satisfying 
recurrence \eqref{eq4} with $|Q| = 1$, $\al > 1$ and suppose that the abc-conjecture holds. If for $n> C_5$, \eqref{eq9a} has infinitely many solutions, then there exist $q \in [2, C_3]$, $r \in \{0, 1,\ldots , q -1\}$,  and integers $a, b, c$  with $(a, b) \neq (0, 0)$, $ \max\left\{|a|, |b|\right\} \leq q\times 6/C_6$ such that infinitely many of these solutions satisfy $n \equiv r \pmod{q}$ and $am +bn =c$.
\end{proposition}

\begin{proof}
Assume that \eqref{eq9a} has infinitely many solutions in the range $n >C_5$, $2 \leq q \leq  C_3$ and  $n - m > C_6n$. We denote these solutions by $(n, m, x_{n, m})$.  
We rewrite \eqref{eq9a} as
 \begin{align}\label{sec4:eq4}
 \begin{split}
 x_{n,m}^q &= U_n+U_m =\frac{a_1\al^n-a_2\be^n}{\sqrt{D}} + \frac{a_1\al^m-a_2\be^m}{\sqrt{D}}\\
 &= \frac{a_1\al^n}{\sqrt{D}}\left( 1+ \al^{m-n} - (-1)^n(\pm 1)^n(a_2/a_1)\al^{-2n} -(-1)^m(\pm 1)^m(a_2/a_1)\al^{-m-n} \right)\\
 &= \frac{a_1\al^n}{\sqrt D} (1+ \mu_{n,m}),
 \end{split}
 \end{align}
 where \[\mu_{n,m} = \al^{m-n} - (-1)^n(\pm 1)^n(a_2/a_1)\al^{-2n} -(-1)^m(\pm 1)^m(a_2/a_1)\al^{-m-n}.\] We set $n = \ell q + r$ with $n\in \mathbb{N}, r\in \{0,\ldots, q-1\}$. Taking $q$-th roots of \eqref{sec4:eq4}, we get
\begin{align}\label{sec4:eq5}
\begin{split}
 x_{n,m} &= \left(\frac{a_1\al^n}{\sqrt D}\right)^{1/q} (1+ \mu_{n,m})^{1/q} = \frac{(a_1\al^r)^{1/q}}{D^{1/2q}}\al^{\ell}(1+ \mu_{n,m})^{1/q}\\
 &= w \al^{\ell}(1+ \mu_{n,m})^{1/q},
 \end{split}
 \end{align}
 where $w= \frac{a_1^{1/q}\al^{r/q}}{D^{1/2q}}$. Notice that $w$  does not depend upon $n$ or $m$. We expand $(1+ \mu_{n,m})^{1/q}$ as
 \begin{equation}\label{sec4:eq6}
 (1+ \mu_{n,m})^{1/q}=\sum_{j=0}^{\infty}\binom{1/q}{j}\mu_{n,m}^j=  \sum_{0\leq j\leq T}\binom{1/q}{j}\mu_{n,m}^j +  \sum_{j> T}\binom{1/q}{j}\mu_{n,m}^j. 
 \end{equation}
Thus, 
\begin{align}\label{sec4:eq7}
\begin{split}
\left|x_{n,m} - \sum_{0 \leq j \leq T }{1/q \choose j} \frac{a_1^{1/q}\al^{n/q}}{D^{1/{2q}}}\mu_{n,m}^j \right| &\leq \sum_{j =T+1}^{\infty} \left| \binom{1/q}{j} \right| \frac{a_1^{1/q}\al^{n/q}}{D^{1/{2q}}} |\mu_{n, m}|^j\\
 & \leq \frac{a_1^{1/q}\al^{n/q}}{D^{1/{2q}}} \frac{1}{q(T+1)}\sum_{j =T+1}^{\infty}|\mu_{n, m}|^j \\
 & =  \frac{a_1^{1/q}\al^{n/q}}{D^{1/{2q}}}\frac{1}{q(T+1)(1-|\mu_{n,m}|)}|\mu_{n, m}|^{T+1}.
 \end{split}
\end{align}
Observe that $ \mu_{n,m} \leq c_{24} \al^{m-n} \leq c_{25}\al^{-c_{14}n}$. So, if $n > \frac{\log c_{25}}{c_{14}\log \al}$ , then  $|c_{25}\al^{-c_{14}n}| < 1$. Thus, from \eqref{sec4:eq7}, we get
\begin{equation}\label{sec4:eq8}
\left|x_{n,m} - \sum_{0 \leq j \leq T }{1/q \choose j} \frac{a_1^{1/q}\al^{n/q}}{D^{1/{2q}}}\mu_{n,m}^j \right| \leq c_{26}\alpha^{(1/q - c_{14}(T+1))n}.
\end{equation}
Now choose \[T> \max \bigg \{1, \frac{1}{c_{14}q} +\frac{2}{c_{14}}-1 \bigg \}.\]
Then $\alpha^{(1/q - c_{14}(T+1))n} < 1/\alpha^{2n}$. Put $\alpha^{-2} =\gamma$. Then for
\begin{equation}\label{sec4:eq9}
n> \frac{\log c_{26}}{2\log \alpha},
\end{equation} we have $c_{26}\alpha^{-2n}< \gamma^n$.
Consequently,
\begin{equation}\label{sec4:eq10}
\left|x_{n,m} - \sum_{0 \leq j \leq T }{1/q \choose j} \frac{a_1^{1/q}\al^{n/q}}{D^{1/{2q}}}\mu_{n,m}^j \right| < \gamma^n, \quad \mbox{with}\quad \gamma<1.
\end{equation} 
We expand each of the $\mu_{n, m}^j$ using the multinomial theorem as
\begin{align}\label{sec4:eq11}
\begin{split}
\mu_{n, m}^j & = \sum_{a+b+c = j} {j \choose a, b, c} (\al^{m-n})^a  ((-1)^m(\pm 1)^m(a_2/a_1)\al^{-m-n})^b ((-1)^n(\pm 1)^n(a_2/a_1)\al^{-2n})^c \\
 &= \sum_{a+b+c = j} {(\mp1)}^{\nu_{a, b, c}} {j\choose a, b, c} (a_2/a_1)^{b+c}\al^{(a-b)m - (a+b+2c)n},
 \end{split}
\end{align}where
\begin{equation}\label{sec4:eq12}
 {(\mp 1)}^{\nu_{a,b,c}}= (\mp1)^{mb+nc}.
 \end{equation} 
 Observe that the exponents of $\al$ in \eqref{sec4:eq11} are of the form $\xi_1m - \xi_2n$, where $\xi_1=a - b$ and $ \xi_2 =a +b +2c $ for some nonnegative integers $a, b, c$ with $a +b +c =j \leq T$. Thus, $|\xi_1| \leq T$ and $0 \leq \xi_2\leq 2T$. Again $\xi_1\leq \xi_2$ always holds and $\xi_1 = \xi_2$ holds if and only if $b =c =0$, it follows that $\xi_1 =\xi_2 = j$. Thus, from \eqref{sec4:eq10} and \eqref{sec4:eq11}, we get
 \begin{equation}\label{sec4:eq13}
 \left|x_{n,m} -w\sum_{\substack{\xi_1 \leq \xi_2 \\ |\xi_1| \leq T \\ 0\leq \xi_2 \leq 2T}} p_{\xi_1, \xi_2} \al^{l + \xi_1m-\xi_2 n} \right|< \gamma^n,
 \end{equation} where the coefficients $p_{\xi_1, \xi_2}$ are given by
 \begin{equation}\label{sec4:eq14}
p_{\xi_1, \xi_2} = \sum_{\substack{(a,b,c)\\ a+b+c \leq T \\ a-b=\xi_1, a+b+2c=\xi_2}} {(\mp1)^{\nu_{a,b,c}}} (a_2/a_1)^{b+c}{1/q \choose a+b+c} {a+b+c \choose a, b, c}.	
\end{equation}

Let $\mathbb{K} = \Q(\sqrt D), \mathbb{L} = \mathbb{K}(D^{1/2q}, a_1\al^{1/q})$. Let $|\cdot |_{v_1}$ and $|\cdot |_{v_2}$ be the two infinite valuations of $\mathbb{K}$ given by $|x |_{v_1} = |\sigma_1(x)|^{1/2}$ and $|x |_{v_2} = |\sigma_2(x)|^{1/2}$ for $x\in \mathbb{K}$, where $\sigma_1(x)$ is the identity automorphism of $\mathbb{K}$ and $\sigma_2(x)$ is the non-trivial automorphism of $\mathbb{K}$ which sends $\sqrt{D}$ to $-\sqrt{D}$. We extend these valuations in some way to $\mathbb{L}$.

We consider the set $\xi = \{(\xi_1, \xi_2) : p_{\xi_1, \xi_2} \neq 0\}$. So if $\xi_1= \xi_2$, then in view of \eqref{sec4:eq14}, we have 
\[ p_{\xi_1, \xi_1}= {(\mp 1)^{\nu_{\xi_1,0,0}}}{1/q \choose \xi_1} \neq 0.\]
Thus, the set $\xi$ contains at least the $T+1$ diagonal pairs $(\xi_1, \xi_1), \xi_1\ge0$. An upper bound of $|\xi|$ is $(2T+1)^2$. We define the linear forms in $1+|\xi|$ variables denoted by
\[\mathtt{x} = (x_0 , x_{(\xi_1, \xi_2)} :( \xi_1, \xi_2) \in \xi)\] in $\mathbb{K}$ given by
\begin{align*}
L_{0,v_1} (\mathtt{x})= x_0 - w \sum_{(\xi_1, \xi_2) \in \xi} p_{\xi_1, \xi_2}  x_{(\xi_1, \xi_2)}, \quad L_{0,v_2}(\mathtt{x}) = x_0
\end{align*} and 
\[L_{(\xi_1, \xi_2), v_i}(\mathtt{x}) = x_{(\xi_1, \xi_2)}, \quad \mbox{for all $(\xi_1, \xi_2) \in \xi$}\; \mbox{and}\;i =1, 2.\]
Then $L_{0,v_i} (\mathtt{x}), L_{(\xi_1, \xi_2), v_i}(\mathtt{x})$ are $1+|\xi|$ linear forms in $1+|\xi|$ variables which are linearly independent for each of $i =1, 2$. From \eqref{sec4:eq13}, we have
\[|L_{0,v_1}(x_{n, m}, \alpha^{l+\xi_1m-\xi_2n})|_{v_1}\leq  c_{27}{\gamma^{n}}\]
and since $x_{n, m} \in \mathbb{Z}$, it follows that
\begin{align*}
|L_{0,v_2}(x_{n, m}, &\alpha^{l+\xi_1m-\xi_2n})|_{v_2} = |x_{n, m}|_{v_2}= |w\alpha^{\ell}(1+\mu_{n, m})^{1/q}|_{v_2}\leq  c_{28}{\alpha^{n/2q}} \leq c_{29}\alpha^{n/4}.
\end{align*}
Further, for $(\xi_1, \xi_2) \in \xi$, we have 
\begin{align*}
|L_{(\xi_1, \xi_2),v_1}&(\al^{l+\xi_1m-\xi_2n})|_{v_1}   |L_{(\xi_1, \xi_2),v_2}(\al^{l+\xi_1m-\xi_2n})|_{v_2} \\
& = |\al\be|^{(l+\xi_1m-\xi_2n)/2} = |Q|^{(l+\xi_1m-\xi_2n)/2} = 1.
\end{align*}
Thus, for $\mathtt{x} = (x_{n, m} , \al^{l+\xi_1m-\xi_2n}: (\xi_1, \xi_2) \in \xi)$,
\begin{align}\label{sec4:eq15}
\begin{split}
& \prod_{i \in {\{1, 2\}}}\prod_{(\xi_1, \xi_2) \in \xi} |L_{0,v_i}(\mathtt{x})|_{v_i} |L_{(\xi_1,\xi_2),v_i}(\mathtt{x})|_{v_i} \\
 &\leq c_{30}\alpha^{-n + n/4}|Q|^{(l+\xi_1m-\xi_2n)/2}\leq c_{30}\alpha^{-3n/4}\cdot 1\\&\leq  c_{30}\alpha^{-3n/4}.
 \end{split}
 \end{align}
 Notice that the coordinates of $\mathtt{x}$ are algebraic integers. Furthermore, $|\xi_1|\leq  T$ and $0\leq \xi_2\leq 2T$ imply $\xi_1m-\xi_2n\leq mT \leq nT$. Hence,
 \begin{align*}
 |\mathtt{x}|_{v_1} &  = \max\{|x_{n, m}|_{v_1}, |\al^{l+\xi_1m-\xi_2n}|_{v_1}: (\xi_1, \xi_2) \in \xi\}\\
 & = \max\{|w\alpha^{\ell}(1+\mu_{n, m})^{1/q}|_{v_1}, |\al^{l+\xi_1m-\xi_2n}|_{v_1}: (\xi_1, \xi_2) \in \xi\}\\
  & \leq c_{31}\alpha^{n(T+1)}.
 \end{align*}
In a similar manner, it can be shown that $|\mathtt{x}|_{v_2} \leq c_{32}\alpha^{n(T+1)}$. Thus, 
 \begin{equation}\label{sec4:eq16}
 H(\mathtt{x}) = \prod_{i\in \{1, 2\}}  |\mathtt{x}|_{v_i} \leq c_{33} \alpha^{2n(T+1)}\leq c_{34} \alpha^{3nT}.
 \end{equation}
  Combining the estimates from \eqref{sec4:eq15} and \eqref{sec4:eq16}, we get 
 \[ \prod_{i \in {\{1, 2\}}} \prod_{(\xi_1, \xi_2) \in \xi} |L_{0,v_i}(\mathtt{x})|_{v_i}  |L_{(\xi_1,\xi_2),v_i}(\mathtt{x})|_{v_i}  < H(\mathtt{x})^{-\delta},\]
provided $\delta < 3/(4c_{14})$.  By Theorem \ref{thmsub} there exist finitely many nonzero rational linear forms $\Lambda_1, \ldots, \Lambda_{s}$ in $\mathbb{K}^{1+|\xi|}$ such that $\mathtt{x}$ is a zero of some $\Lambda_j$.

{\sc Case I:} Suppose $\Lambda_j$ does not depend on $x_0$. In this case, there are 
 some fixed coefficients $g_{(\xi_1, \xi_2)}$ for $(\xi_1, \xi_2) \in \xi$, not all zero, such that infinitely many solutions satisfy 
 \[\sum_{(\xi_1, \xi_2) \in \xi} g_{(\xi_1, \xi_2)}\al^{l+\xi_1m- \xi_2n} = 0.\]
 This is an $S$-unit equation, where $S$ is the multiplicative group generated by $\al$. If it does have zero sub-sums, then we can replace it with a minimal such one, which is a zero sub-sum and has no non-trivial zero sub-sum. By Theorem \ref{prop4}, each one has only finitely many projective solutions. That is, if $(\xi_1, \xi_2) \neq (\xi_1', \xi_2')$ such that \[ g_{(\xi_1, \xi_2)}\al^{l+\xi_1m- \xi_2n} +  g_{(\xi_1', \xi_2')}\al^{l+\xi_1'm- \xi_2'n}= 0,\] then the ratio
\[\frac{\al^{\xi_1m-\xi_2n}}{\al^{\xi_1'm- \xi_2'n}} \]
belongs to a finite set. Thus, there are only finite number of values of $(\xi_1 -\xi_1' )m -(\xi_2 -\xi_2' )n$. Since $|\xi_1 -\xi_1'| \leq 2T \leq 2(3/c_{14})$ and $|\xi_2 -\xi_2'| \leq 2T \leq 2(3/c_{14})$, it follows that there are infinite number of pairs $(m, n)$ satisfying \eqref{eq9a} for fixed $q$ and $n \equiv r \pmod{q}$ such that $am + bn =c$ holds for some integers $(a, b, c)$ with $(a, b) \neq (0, 0)$ and $\max\{|a|, |b|\} \leq 6/c_{14}$.

{\sc Case II:} Suppose  $\Lambda_j$ depends on $x_0$. Clearly, $\Lambda_j \neq 0$ as our solutions contain  $x_{n,m} > 0$ and $\Lambda_j$ depends on $x_0$. So, there are coefficients $g_{\xi_1, \xi_2}$ for $(\xi_1, \xi_2) \in \xi$, not all zero and the subspace is given by the equation 
 \[\sum_{(\xi_1, \xi_2) \in \xi} g_{(\xi_1, \xi_2)}\al^{l+\xi_1m- \xi_2n} = x_0.\]
Let $\xi_{1}^0$ be the minimal of all $\xi_1$ with $g_{\xi_{1}^0,\xi_2} \neq 0$ for some $\xi_2$. Fixing $\xi_{1}^0$, let $\xi_{2}^0$ be the maximal of all $\xi_2$ such that $g_{(\xi_{1}^0,\xi_{2}^0)} \neq 0$. So  the minimum of the expression $\xi_1m - \xi_2n$ over all $(\xi_1, \xi_2)$ with $g_{(\xi_{1}^0,\xi_{2}^0)} \neq 0$ is $(\xi_{1}^0,\xi_{2}^0)$. Thus, 
\begin{equation}\label{sec4:eq17}
x_0^q = g_{(\xi_1^{0},\xi_2^{0})}^q\al^{lq+q(\xi_1^0m-\xi_2^0n)} + \sum_{(\xi_1^1,\xi_2^1), \ldots, (\xi_1^q,\xi_2^q)}' g_{(\xi_1^1,\xi_2^1)} \cdots g_{(\xi_1^q,\xi_2^q)}\al^{lq+ q(\sum_{t=1}^{q}\xi_1^t)m -  q(\sum_{t=1}^{q}\xi_2^t)n},
\end{equation}
 where $'$ means that in the summation range, the option $(\xi_1^1,\xi_2^1)= \cdots =(\xi_1^q,\xi_2^q) = (\xi_1^0,\xi_2^0)$ is excluded. Again since,
\[x_0^q = U_n+ U_m = \frac{a_1}{\sqrt D}(\al^n+ \al^{m} - (-1)^n (\pm 1)^n(a_2/a_1)\al^{-n} - (-1)^m(\pm 1)^m(a_2/a_1)\al^{-m}),\] 
we get an $S$-unit equation
\begin{align}\label{sec4:eq18}
\begin{split}
&\frac{a_1}{\sqrt D}( \al^n+ \al^{m} - (-1)^n (\pm 1)^n(a_2/a_1)\al^{-n} - (-1)^m(\pm 1)^m(a_2/a_1)\al^{-m}) - g_{(\xi_1^{0} \xi_2^{0})}^q\al^{lq+q(\xi_1^0m-\xi_2^0n)} \\
&- \sum_{(\xi_1^1,\xi_2^1), \ldots, (\xi_1^q,\xi_2^q)}' g_{(\xi_1^1,\xi_2^1)} \cdots g_{(\xi_1^q,\xi_2^q)}\al^{lq+ q(\sum_{t=1}^{q}\xi_1^t)m -  q(\sum_{t=1}^{q}\xi_2^t)n} =0.
\end{split}
\end{align}
Consider a non-degenerate sub-sum of \eqref{sec4:eq18} as
\[g_{(\xi_1^{0}, \xi_2^{0})}^q\al^{lq+q(\xi_1^0m-\xi_2^0n)} +  g_{(\xi_1^1,\xi_2^1)} \cdots g_{(\xi_1^q,\xi_2^q)}\al^{lq+ q(\sum_{t=1}^{q}\xi_1^t)m -  q(\sum_{t=1}^{q}\xi_2^t)n} = 0.\] By Theorem \ref{prop4}, 
\[\frac{\al^{lq+q(\xi_1^0m-\xi_2^0n)}}{\al^{lq+ q(\sum_{t=1}^{q}\xi_1^t)m -  q(\sum_{t=1}^{q}\xi_2^t)n}}\] belongs to a finite set. Thus, 
\begin{equation}\label{sec4:eq19}
\left(\xi_1^0 -\sum_{t=1}^{q}\xi_1^t\right)qm - \left(\xi_2^0 -\sum_{t=1}^{q}\xi_2^t\right)qn
\end{equation} belongs to a finite set. One can observe that at least one of the coefficients of $m$ or $n$ in \eqref{sec4:eq19} is nonzero. Hence, we get a relation of the form $am + bn =c$, with $(a, b) \neq (0, 0)$, $\max\{|a|, |b|\} \leq q\times 6/c_{14}$ and finitely many possibilities for $c$.

Next consider  a non-degenerate sub-sum of \eqref{sec4:eq18} containing $\al^{n+q(\xi_1^0m-\xi_2^0n)}$ and any one of $\alpha^n, \alpha^{-n}, \alpha^m$ or $\alpha^{-m}$. For example,  
\[g_{(\xi_{1}^0,\xi_{2}^0)}\al^{lq+q(\xi_1^0m-\xi_2^0n)}  + \alpha^n = 0.\] Then $\al^{q(\xi_1^0m-\xi_2^0n)}$ belongs to a finite set and hence, we get a relation of the form $am + bn =c$. Similarly, any one of these possibilities gives a relation of the form $am + bn =c$ with $(a, b) \neq (0, 0)$, $\max\{|a|, |b|\} \leq \frac{6q}{c_{14}}$ and finitely many possibilities for $c$. This completes the proof of proposition.
\end{proof}

 
 {\sc Proof of Theorem\ref{th3}} Suppose that the {\em abc-conjecture} holds and $n> C_5, q> C_3$ and $n-m < C_6n$. Then, by Proposition \ref{propo1}, \eqref{eq9a} has only finitely many solutions. 
 
 Now, we may assume that $2\leq q\leq C_3, \; n-m > C_6n$. On contrary, suppose that  \eqref{eq9a} has infinitely many solutions. Then, by Proposition \ref{propo2}, there are integers $a, b$ and $c$ with $(a, b) \neq (0, 0), \max\{|a|, |b|\} \leq \frac{ 6q} {C_6}$ such that infinitely many of these solutions satisfy $n\equiv r\pmod{q}$ and $am + bn = c$.
 
Now, assume that $a$ and $b$ are coprime. If $a = 0$, then $n$ is bounded, which is a contradiction. If $ b = 0$, then $m = c/a$, which is a fixed number. So, putting $t_0= U_m$ in \eqref{eq9a}, we have that 
\begin{equation}\label{v5eqpage17}
U_n = x^q - t_0
\end{equation}
has infinitely many positive solutions in $n$ and $x$.  In view of \cite[Theorem 3]{petho86},  $U_n = x^q-t_0$ has only finitely many solutions in $n$ and $x$ with $q>2$, which is a contradiction. Thus, $b\neq 0$ for $q>2$. Next suppose that $b =0$ and $q = 2$. That is,  $U_n = x^2 - t_0$ has infinitely many solutions in $n$ and $x$. Since $(U_n)_{n\geq 0}$ is non-degenerate, by \cite[Theorem 1]{petho86}, $x^2-t_0 = \pm \sqrt{\frac{a_1a_2}{D}}T_2\left( \pm \sqrt{\frac{2}{U_m}} x\right)$. Comparing the leading coefficients and substituting $U_m =t_0$, we get $U_m = \pm 2\sqrt{\frac{a_1a_2}{D}}$, that is, $U_m^2= 4a_1a_2/D$, which is not true by our assumption. Thus,  $ab\neq 0$. Note that $a$ and $b$ must have opposite sign, otherwise we get only finitely many solutions. Changing the sign of $b$ to $-b$, assume that $am - bn =c$, where $a$ and $b$ are positive integers. Then $n = n_0 + a t, \; m=m_0 + b t$ for some integer $t$ and some fixed integers $n_0,m_0$ satisfying $am_0 - bn_0=c$. For large $m\; \hbox{and} \; n$, $t$ is positive and since $n-m > c_1n$, it follows that $ b < a$. Thus, 
\begin{align}\label{sec4:eq20}
\begin{split}
x^q & = U_n + U_m \\
& = \frac{1}{\sqrt D}\left( \al^{n_0+ at}+ \al^{m_0 + bt} - (-1)^n(\pm 1)^{n_0+ at}\al^{-n_0- at} - (-1)^m(\pm 1)^{m_0 + bt}\al^{-m_0 - bt}\right).
\end{split}
\end{align}
Note that $n = n_0 + at = \ell q +r$, and therefore, $at- n_0 = \ell q + (r - 2n_0)$. Now multiplying $\alpha^{at- n_0}$ on both sides of \eqref{sec4:eq20}, we get
\begin{align*}
\begin{split}
 \delta y^q=\al^{2at}+ \alpha^{(a+b)t}\al^{m_0}\alpha^{-n_0} - (-1)^n(\pm 1)^{n_0+ at}\al^{-2n_0} - (-1)^m (\pm 1)^{m_0 + bt}\alpha^{(a-b)t- m_0 - n_0},
\end{split}
\end{align*}
where $\delta = \sqrt{D}\alpha^{r-2n_0}$ and $y = \alpha^{\ell}x \in \mathcal{O}_{\mathbb{K}}$. Hence, the polynomial 
\[P(X) = X^{2a}+ X^{a+b}\al^{m_0} - (-1)^m(\pm 1)^{b+m_0}X^{a-b}  - (-1)^n(\pm 1)^{a+n_0}\al^{-n_0} \in \mathcal{O_\mathbb{K}}[X]\] is such that the equation $P(x_1) = \delta y^q$ has infinitely many solutions $(x_1, y)\in \mathcal{O}_{\mathbb{K}}^2$. Now using Theorem \ref{siegel} and the arguments in similar context from \cite{fujita2017, kkll}, we conclude that all roots of $P(X)$ have multiplicities which are multiples of $p$, where $p$ is a prime divisor of $q$.  Hence,
\begin{equation}\label{sec4:eq21}
P(X) = Q_1(X)^p \in \mathbb{K}[X]\end{equation} since $P(X)$ is monic polynomial. Again,  $\mathcal{O_\mathbb{K}}[X]$ is integrally closed and $P(X) \in \mathcal{O_\mathbb{K}}[X]$, so $Q_1(X) \in \mathcal{O_\mathbb{K}}[X]$. Now, putting $e_1 := \deg Q_1(X)$ and then comparing degrees of both polynomials in \eqref{sec4:eq21}, we get $e_1p =2a$, so $p\mid 2a$ and hence, $e_1 =2a/p$. Now write $Q_1(X) =X^{e_1}+ \ga_1 X^{e_1 - e_2}+ \mbox{monomials of degree less than $e_1- e_2$}$, where $e_2\geq 1$. Then 
\begin{align*}
P(X) & = X^{2a}+ X^{a+b}\al^{m_0} - (-1)^m(\pm 1)^{b+m_0}X^{a-b}  - (-1)^n(\pm 1)^{a+n_0}\al^{-n_0}\\
& = Q_1(X)^p = \left(X^{e_1}+ \ga_1 X^{e_1 - e_2}+ \cdots \right)^p\\
& = X^{pe_1} + p\ga_1 X^{pe_1-e_2}+ \cdots 
\end{align*} 
Comparing the second leading coefficients, we get $\ga_1 = \alpha^{m_0}/p$, which is not an algebraic integer because $p \geq 2$ is a prime and $\al$ is a unit, that is $Q_1[X] \not \in \mathcal{O_\mathbb{K}}[X]$. This is a contradiction as $Q_1(X) \in \mathcal{O_\mathbb{K}}[X]$. This completes the proof. \qed

\section{Concluding remark}
Notice that the proof of Theorem \ref{th3} depends on Proposition \ref{propo1}. Further, to prove Proposition \ref{propo1}, we assume that $|Q| =1$ and the {\em abc-conjecture}. In Proposition \ref{propo1}, we are unable to provide a bound for $n$ when $|Q| \neq 1$. It will be nice to prove Theorem \ref{th3} without assuming $|Q| =1$ and the {\em abc-conjecture}.


{\bf Acknowledgment:} We thank Professor Attila Peth\"{o} for providing a copy of his paper \cite{petho86}. The first author's work is supported by CSIR fellowship(File no: 09/983(0036)/2019-EMR-I) and S. S. Rout wants to thank SERB for the support  (File no.: CRG/2022/000268).

\end{document}